\documentclass[a4j,12pt]{article}
\usepackage[a4paper,vmargin={1in,1.in},hmargin={1in,1in}]{geometry}
\usepackage{amsmath}
\usepackage{amssymb}
\usepackage{amsthm}
\usepackage{amscd}
\usepackage{sectsty}
\usepackage[dvipdfm]{graphicx}

\DeclareGraphicsExtensions{.pdf,.png,.jpeg,.gif}
\numberwithin{equation}{section}
\theoremstyle{definition}
\newtheorem{example}{Example}[section]
\newtheorem{definition}[example]{Definition}
\theoremstyle{plain}
\newtheorem{theorem}[example]{Theorem}
\theoremstyle{plain}
\newtheorem{corollary}[example]{Corollary}
\theoremstyle{remark}

\theoremstyle{lemma}
\newtheorem{lemma}[example]{Lemma}
\theoremstyle{problem}

\theoremstyle{proposition}
\newtheorem{proposition}[example]{Proposition}
\theoremstyle{question}

\theoremstyle{conjecture}

\theoremstyle{theorem}
\newtheorem*{theoremnn*}{Theorem A}
\theoremstyle{corollarynn}
\newtheorem*{corollarynn*}{Theorem B}
\theoremstyle{corollarynnn}
\newtheorem*{corollarynnn*}{Theorem C}
\theoremstyle{definition}
\newtheorem*{definitionnn*}{Definition}

\begin{document}
\sectionfont{\fontsize{12}{15}\selectfont}
\subsectionfont{\fontsize{12}{15}\selectfont}

\begin{center}
{\Large \textbf{$J$-Stability of immediately expanding polynomial maps in $p$-adic dynamics}} \\\

LEE, Junghun
\footnote{
{\it Graduate School of Mathematics, Nagoya University, 
Nagoya 464-8602, Japan \\
(e-mail: m12003v@math.nagoya-u.ac.jp)}}
\end{center}
\begin{abstract}
Given a family $\{ f_{\lambda} \}_{\lambda \in \Lambda}$ of polynomial maps of degree $d$ where $\Lambda$ is the set of parameters, a polynomial map $f_{\lambda_0}$ is called {\it $J$-stable in $\Lambda$} if there exists a neighborhood of $\lambda_0$ in $\Lambda$ such that for any element $\lambda$ in the neighborhood, there exists a conjugacy between the dynamics on the Julia sets of $f_\lambda$ and $f_{\lambda_0}$.
The aim of this paper is to show that a polynomial map $f_{\lambda_0}$ over the field $\mathbb{C}_p$ of $p$-adic complex numbers is $J$-stable in the family of polynomial maps over $\mathbb{C}_p$ if $f_{\lambda_0}$ is {\it immediately expanding}.
\end{abstract}

\section{Introduction}

In the theory of complex dynamical systems, we consider the iterations $\{ f^{n} \}_{n \in \mathbb{N}}$ of a rational map $f : \hat{\mathbb{C}} \rightarrow \hat{\mathbb{C}}$ over the field $\mathbb{C}$ of complex numbers where $\hat{\mathbb{C}}$ is the Riemann sphere and $f^n$ denotes the $n$ th iteration of $f$.
The theory of complex dynamical systems was established by P. Fatou and G.
Julia in the early 20th century and they considered the Fatou set $\mathcal{F}(f)$ of $f$, which is defined as the largest open set on which the family $\{ f^n \}_{n \in \mathbb{N}}$ of iterated rational maps is equicontinuous, and the Julia set $\mathcal{J}(f)$ of $f$, which is a compact set defined as the complement of $\mathcal{F}(f)$.
These two notions, the Fatou set $\mathcal{F}(f)$ and the Julia set $\mathcal{J}(f)$, are essential in the theory of complex dynamical systems because the Fatou set $\mathcal{F}(f)$ and the Julia set $\mathcal{J}(f)$ are the stable region and the chaotic locus of dynamics, respectively, and both of them are completely invariant under $f$.
See \cite{Miln06} for more details on the theory of complex dynamical systems.

In the theory of $p$-adic dynamical systems, we consider the iterations $\{ f^n \}_{n \in \mathbb{N}}$ of a rational map $f : \hat{\mathbb{C}}_p \rightarrow \hat{\mathbb{C}}_p$ over the field $\mathbb{C}_p$ of $p$-adic complex numbers where $\hat{\mathbb{C}}_p$ is the projective line over $\mathbb{C}_p$.
As we do in the theory of complex dynamical systems, we also define the Fatou set $\mathcal{F}(f)$ of $f$ and the Julia set $\mathcal{J}(f)$ of $f$ with equicontinuity.
We also obtain the invariance of $\mathcal{F}(f)$ and $\mathcal{J}(f)$ under $f$. 
See \cite{Hs00} or \cite{Silv07} for more details.
Unlike the Riemann sphere $\hat{\mathbb{C}}$, the projective line $\hat{\mathbb{C}}_p$  over $\mathbb{C}_p$ has different topological properties such as non-compactness and disconnectedness. These also effect on the dynamics. 
For example, the Julia set of a polynomial map might be non-compact.

The theory of $p$-adic dynamical systems is relatively new, and mostly developed in this century.
For example, L-C. Hsia proved a $p$-adic analogue of Montel's theorem in \cite{Hs00}. 
R. Benedetto proved an analogue of Sullivan's no wandering domain theorem in \cite{Ben00}. He also found a polynomial map over $\mathbb{C}_p$ which has a wandering domain in \cite{Ben02}.
Moreover, J. Rivera-Letelier developed the theory of $p$-adic hyperbolic space in \cite{Riv03a} and \cite{Riv03b}.

There is a natural question in the theory of dynamical systems: are there any relations of the Julia sets of two maps if those two maps are close enough?
In the theory of complex dynamical systems, Man\~e-Sad-Sullivan proposed an answer to this question in \cite{MSS83}. We will see their theorem and its applications in subsection $1.1$ as a motivation of this paper.
However, in the theory of $p$-adic dynamical systems, there was no analogue of Man\~e-Sad-Sullivan's.
Our aim of this paper is to give an analogue of Man\~e-Sad-Sullivan's theorem in $p$-adic dynamics (Theorem \ref{2.1}). 
Moreover, we will apply it to the family $\{ z^d + c \}_{c \in \mathbb{C}_p}$ of polynomial maps of degree $d$ over $\mathbb{C}_p$ in $p$-adic dynamics (Theorem \ref{2.3}) when $d$ is not divided by $p$.

\subsection{Motivations from the theory of complex dynamical systems}


In the theory of complex dynamical systems, Man\~e-Sad-Sullivan established the notion of $J$-stability and proved the following theorem in \cite{MSS83}.
We shall denote {\it the Riemann sphere} by $\hat{\mathbb{C}} : = \mathbb{C} \cup \{ \infty \}$ where $\infty$ is a symbol which is not contained in $\mathbb{C}$.

\begin{theoremnn*}[Man\~e-Sad-Sullivan]
Let $f: \hat{\mathbb{C}} \rightarrow \hat{\mathbb{C}}$ be a rational map of degree $d \geq 2$ over $\mathbb{C}$.
Suppose that there exists a connected open neighborhood $U$ of $f$ in the family of rational maps of degree $d$ over $\mathbb{C}$ such that for any element $g$ in $U$, the number of attracting cycles of $g$ is equal to the number of attracting cycles of $f$.
Then for any $g$ in $U$, there exists a homeomorphism $h : \mathcal{J}(f) \rightarrow \mathcal{J}(g)$ such that 
$
h \circ f = g \circ h
$
on the Julia set $\mathcal{J}(f)$ of $f$.
\end{theoremnn*}

We say such a rational map $f : \hat{\mathbb{C}} \rightarrow \hat{\mathbb{C}}$ is {\it $J$-stable in the family of rational maps}.

The following theorem is an application of Theorem A to hyperbolic polynomial maps.
Recall that a polynomial map $f : \hat{\mathbb{C}} \rightarrow \hat{\mathbb{C}}$ over $\mathbb{C}$ is {\it hyperbolic} if there exist a $\lambda > 1$ and a $c > 0$ such that $|(f^{n})'(z)| \geq c \cdot \lambda^n$ for any $z$ in the Julia set $\mathcal{J}(f)$ of $f$ and $n$ in $\mathbb{N}$.

\begin{corollarynn*}
Let $f : \hat{\mathbb{C}} \rightarrow \hat{\mathbb{C}}$ be a polynomial map of degree $d \geq 2$ over $\mathbb{C}$.
If $f$ is hyperbolic,
then there exists a connected open neighborhood $U$ of $f$ in the family of polynomial maps of degree $d$ over $\mathbb{C}$ such that for any $g$ in $U$, there exists a homeomorphism $h : \mathcal{J}(f) \rightarrow \mathcal{J}(g)$ such that 
$
h \circ f = g \circ h
$
on the Julia set $\mathcal{J}(f)$ of $f$.
\end{corollarynn*}

We say such a polynomial map $f : \hat{\mathbb{C}} \rightarrow \hat{\mathbb{C}}$ is {\it $J$-stable in the family of polynomial maps}.
Let us see an application of Theorem A to the family $\{ z^d + c \}_{c \in \mathbb{C}}$ of polynomial maps of degree $d \geq 2$ over $\mathbb{C}$.

\begin{corollarynnn*}\label{corC}
Let $d$ be a natural number with $d \geq 2$ and consider the polynomial maps 
\begin{align*}
f_{ c}(z) := 
\begin{cases}
z^d + c &(z \in \mathbb{C}), \\
\infty &(z = \infty).
\end{cases}
\end{align*}
where $c$ in $\mathbb{C}$.
Suppose that the parameters $c$ and $c'$ satisfy
$$
\lim_{k \rightarrow \infty} | f_{c}^{k} (0)| = \lim_{k \rightarrow \infty} | f_{c'}^{k}(0)| = \infty.
$$
Then there exists a homeomorphism $h : \mathcal{J}(f_c) \rightarrow \mathcal{J}(f_{c'})$ such that $h \circ f_{ c} = f_{ c'} \circ h$ on the Julia set $\mathcal{J}(f_{ c})$ of $f_c$.
\end{corollarynnn*}

Since the set $\{c \in \mathbb{C} \mid \lim_{k \rightarrow \infty} | f^{k}_c(0)| = \infty \}$ where $f_c : \hat{\mathbb{C}} \rightarrow \hat{\mathbb{C}}$ is the polynomial map defined in Theorem C is open and connected in $\mathbb{C}$ and every $f_c$ whose parameter is in $\{c \in \mathbb{C} \mid \lim_{k \rightarrow \infty} | f^{k}_c(0)| = \infty \}$ has only one attracting cycle at infinity, the proof of Theorem C follows immediately from Theorem A.

\subsection{The main theorems}

Let us begin with the definition of {\it $J$-stability in the family of polynomial maps} of $p$-adic dynamics, which is an analogue of $J$-stability in the family of polynomial maps of complex dynamics.
We shall use $\mathbb{C}_p$, $|\cdot|_p$, and $Poly_d(\mathbb{C}_p)$ to denote {\it the field of $p$-adic complex numbers}, {\it the $p$-adic norm on $\mathbb{C}_p$}, and {\it the family of polynomial maps of degree $d \geq 1$}, respectively.
In particular, we consider $Poly_{d}(\mathbb{C}_p)$ as a topological space with the topology of $\mathbb{C}_p^{\times} \times \mathbb{C}_p^{d}$.
We shall also denote {\it the projective line over $\mathbb{C}_p$} by $\hat{\mathbb{C}}_p : = \mathbb{C}_p \cup \{ \infty \}$ where $\infty$ is a symbol which is not contained in $\mathbb{C}_p$.
See section $2$ for more details.

\begin{definitionnn*}
Let $p$ be a prime number, $f : \hat{\mathbb{C}}_p \rightarrow \hat{\mathbb{C}}_p$ be a polynomial map of degree $d \geq 2$ over $\mathbb{C}_p$, and $A$ be a set in $\mathbb{C}_p$.

\begin{enumerate}

\item
We say that $f$ is {\it $J$-stable in $Poly_d(\mathbb{C}_p)$} if there exists a neighborhood $U$ of $f$ in $Poly_d(\mathbb{C}_p)$ such that for any $g$ in $U$, there exists a homeomorphism $h: \mathcal{J}(f) \rightarrow \mathcal{J}(g)$ such that
$$
h \circ f = g \circ h 
$$ 
on the Julia set $\mathcal{J}(f)$ of $f$.

\item
We say that $f$ is {\it immediately expanding on $A$} if there exists a $\lambda > 1$ such that for any $z$ in $A$, $|f'(z)|_p \geq \lambda$. 
In particular, we say that $f$ is {\it immediately expanding} if the Julia set $\mathcal{J}(f)$ of $f$ is non-empty and $f$ is immediately expanding on $\mathcal{J}(f)$.
\end{enumerate}

\end{definitionnn*}

The aim of this paper is to prove the following theorem that is an analogue of Theorem B in $p$-adic dynamics.

\begin{theorem}\label{2.1}
Let $p$ be a prime number and $f : \hat{\mathbb{C}}_p \rightarrow \hat{\mathbb{C}}_p$ be a polynomial map of degree $d \geq 2$ over $\mathbb{C}_p$.
Suppose that $f$ is immediately expanding.
Then $f$ is $J$-stable in $Poly_d(\mathbb{C}_p)$.
\end{theorem}

We will also show the following theorem to prove Theorem \ref{2.1}.
We shall use $\overline{D}_r(a)$ to denote the set $\{ z \in \mathbb{C}_p \mid |z - a|_p \leq r \}$ for an $a \in \mathbb{C}_p$ and an $r > 0$.

\begin{theorem}\label{2.2}
Let $p$ be a prime number and $f : \hat{\mathbb{C}}_p \rightarrow \hat{\mathbb{C}}_p$ be a polynomial map of degree $d \geq 2$ over $\mathbb{C}_p$.
Suppose that the Julia set $\mathcal{J}(f)$ of $f$ is non-empty and there exists a non-empty set $B$ in $\mathbb{C}_p$ satisfying the following conditions.
\begin{enumerate}
\item The polynomial map $f$ is immediately expanding on $B$. 
\item The set $B$ is backward invariant under $f$, that is, $f^{-1}(B) \subset B$.
\item There exists a positive real number $M$ such that $B \subset \overline{D}_{M}(0)$.
\item There exists a positive real number $\delta$ such that $\overline{D}_{\delta}(z) \subset B$ for any $z \in B$.
\end{enumerate}
Then, $f$ is $J$-stable in $Poly_d(\mathbb{C}_p)$.
\end{theorem}

As a corollary of Theorem \ref{2.1}, we also have the following theorem which is an analogue of Theorem C in $p$-adic dynamics.

\begin{theorem}\label{2.3}
Let $p$ be a prime number, $d$ be a natural number with $d \geq 2$, and consider the polynomial maps
\begin{align*}
f_{c}(z) := 
\begin{cases}
z^d + c &(z \in \mathbb{C}_p), \\
\infty &(z = \infty).
\end{cases}
\end{align*}
where $c$ in $\mathbb{C}_p$.
Suppose that $d$ is not divided by $p$ and the parameters $c$ and $c'$ satisfy
$$
\lim_{k \rightarrow \infty} | f_{c}^{k}(0) |_p = \lim_{k \rightarrow \infty} |f_{c'}^{k}(0)|_p = \infty \quad \text{and} \quad |c - c'|_p \leq |c|_p^{1 / d}.
$$
Then there exists a homeomorphism $h : \mathcal{J}(f_c) \rightarrow \mathcal{J}(f_{c'})$ such that $h \circ f_{ c} = f_{ c'} \circ h$ on the Julia set $\mathcal{J}(f_{ c})$ of $f_c$.\end{theorem}

One of the differences between Theorem \ref{2.3} and Theorem C is the condition that two parameters must be close enough.
In fact, this condition can be ignored in the theory of complex dynamical systems because the set $\{ c \in \mathbb{C} \mid \lim_{k \rightarrow \infty} |f_c^k(0)| = \infty \}$ where $f_c$ is the polynomial map over $\mathbb{C}$ defined in Theorem C is path connected.
See \cite[Theorem 9.10.2]{Beard00}.
However, in the theory of $p$-adic dynamical systems, the set $\{ c \in \mathbb{C}_p \mid \lim_{k \rightarrow \infty} |f_c^k(0)|_p = \infty \}$ where $f_c$ is the polynomial maps over $\mathbb{C}_p$ defined in Theorem \ref{2.3} is totally disconnected.

\subsection*{Contents of this paper}

In the second section, we will define some notations and recall a primer on a construction of the field of $p$-adic complex numbers (subsection $2.1$), the theory of $p$-adic dynamical systems (subsections $2.2$ and $2.3$), and $p$-adic analysis (subsection $2.4$).

In the third section, we will see a key lemma and its proof (subsection $3.1$).
Moreover, we will prove Theorem \ref{2.1} (subsection $3.2$), Theorem \ref{2.2} (subsection $3.3$), and Theorem \ref{2.3} (subsections $3.4$).

In the last section, we recall the symbolic dynamical system (subsection $4.1$) and consider the correspondence between dynamics of some polynomial maps and the symbolic dynamical system as an application of the main theorems (subsection $4.2$).

\section{A primer on $p$-adic dynamical systems}

In this section, we will see a construction of the field of $p$-adic complex numbers, and the projective line over it with the chordal metric to consider $p$-adic dynamics.
We also see some facts of $p$-adic analysis. 

\subsection{The field of $p$-adic complex numbers}

\begin{definition}
Let $p$ be a prime number. 
We define {\it the $p$-adic norm on $\mathbb{Q}$} by
\begin{align*}
\left| {m \over n } \right|_p :=
\begin{cases}
0 \quad &(m = 0), \\
p^{-k} \quad &(m \neq 0)
\end{cases}
\end{align*}
where $k$ is the natural number satisfying 
$$
{m \over n} = p^{k} {m' \over n'} \quad \text{and} \quad p \nmid m', n'.
$$
\end{definition}

Considering the completion of the algebraic closure of the completion of $\mathbb{Q}$ with respect to the $p$-adic norm $|\cdot|_p$, we obtain an algebraically closed, complete field of characteristic zero. 
Let us denote it by  $\mathbb{C}_p$ and call $\mathbb{C}_p$ {\it the field of $p$-adic complex numbers}.
Moreover, the $p$-adic norm on $\mathbb{Q}$ is uniquely extended to $\mathbb{C}_p$.
We denote the extended norm on $\mathbb{C}_p$ by $|\cdot|_p$ and call it {\it the $p$-adic norm on ${\mathbb{C}}_p$}.

One of the most important properties of the $p$-adic norm on $\mathbb{C}_p$ is as follows.

\begin{proposition}\label{nonarchi}
Let $p$ be a prime number. 
For any $z$ and $w$ in $\mathbb{C}_p$, we have the inequality $|z \pm w|_p \leq \max\{|z|_p, |w|_p \}$.
Moreover, if $z$ and $w$ satisfy $|z|_p \neq |w|_p$, then $|z \pm w|_p = \max\{|z|_p, |w|_p \}$.
\end{proposition}

We refer \cite{Rob00} to the readers for more details on the $p$-adic norm on $\mathbb{C}_p$. 

\subsection{The projective line over $\mathbb{C}_p$ and the chordal metric}
\begin{definition}
Let $p$ be a prime number and $\infty$ be a symbol called {\it infinity}. We define {\it the projective line over $\mathbb{C}_p$} by 
$$
\hat{\mathbb{C}}_p := \mathbb{C}_p \cup \{ \infty \}.
$$
Moreover, we define {\it the chordal metric on $\hat{\mathbb{C}}_p$} by 
\begin{align*}
\rho_p(z, w) :=
\begin{cases}
\displaystyle{|z - w |_p \over \max\{ |z|_p,1 \} \cdot \max\{ |w|_p, 1 \} } \quad &(z, w \in \mathbb{C}_p), \\
\displaystyle{1 \over \max\{ |z|_p,1 \} } \quad &(z \in \mathbb{C}_p, \quad w = \infty), \\
0 \quad &(z = w = \infty).
\end{cases}
\end{align*}
\end{definition}

One can check that the chordal metric $\rho_p$ satisfies the axioms of metric. See \cite[Proposition$2.4$]{Silv07} for more details.

Let us fix a prime number $p$ and use the following notations throughout this paper. 
\begin{align*}
&|\cdot| = |\cdot|_p, \quad \rho = \rho_p, \\
&{D}_r(a) = \{ z \in \mathbb{C}_p \mid |z - a|_p < r \}, \\
&\overline{D}_r(a) = \{ z \in \mathbb{C}_p \mid |z - a|_p \leq r \}, \\
&\mathbb{C}_p^{\times} = \mathbb{C}_p \backslash \{ 0 \}, \\
&|\mathbb{C}_p^{\times}| = \{ |z|_p \in \mathbb{R} \mid z \in \mathbb{C}_p^{\times} \}.
\end{align*}

In particular, we call a set $A$ in $\mathbb{C}_p$ {\it a closed (resp. open) disk of $\mathbb{C}_p$} if there exist an $a \in \mathbb{C}_p$ and an $r > 0$ such that $A = \overline{D}_r(a)$ (resp. ${D}_r(a))$.
Let us recall some properties of disks of $\mathbb{C}_p$:
The first property is that {\it the closed unit disk} $\overline{D}_1(0)$ is open, closed, and non-compact in $\mathbb{C}_p$.
In particular, this implies that $\hat{\mathbb{C}}_p$ is totally disconnected and non-compact.
The second property is that if two distinct disks in $\mathbb{C}_p$ have a non-empty intersection, one of the disks must be contained in the other. 
In particular, this implies that the intersection of two disks in $\mathbb{C}_p$ is also a disk in $\mathbb{C}_p$.
The third property is that every bounded sets in $\mathbb{C}_p$ is closed with respect to $|\cdot|$ if and only if it is closed with respect to $\rho$.

\subsection{The Fatou set and the Julia set}

Given a polynomial map $f : \hat{\mathbb{C}}_p \rightarrow \hat{\mathbb{C}}_p$ with $f(\infty) = \infty$ over $\mathbb{C}_p$, we define {\it the Fatou set $\mathcal{F}(f)$ of $f$} by the largest open subset of $\hat{\mathbb{C}}_p$ on which the family $\{ f^n \}_{n \in \mathbb{N}}$ of iterated polynomial maps is equicontinuous.
Moreover, we call the complement of $\mathcal{F}(f)$ {\it the Julia set $\mathcal{J}(f)$ of $f$}.
Note that the Fatou set of a polynomial map contains $\infty$ so the Julia set must be bounded.
We refer \cite{Silv07} to readers for a general definition of the Fatou set and the Julia set in $p$-adic dynamical systems.
It is well-known that $\mathcal{F}(f)$ and $\mathcal{J}(f)$ are {\it totally invariant under $f$}, that is, 
$$
f(\mathcal{F}(f)) = f^{-1}(\mathcal{F}(f))= \mathcal{F}(f) \quad \text{and} \quad f(\mathcal{J}(f)) = f^{-1}(\mathcal{J}(f)) = \mathcal{J}(f).
$$
Moreover, one can show that 
$$
\mathcal{F}(f) = \mathcal{F}(f^k) \quad \text{and} \quad \mathcal{J}(f) = \mathcal{J}(f^k)
$$
for any $k \in \mathbb{N}$.
See \cite[Proposition $5.18$]{Silv07} for the proof.

Unlike the theory of complex dynamical systems, we can easily find an example of a polynomial map whose Julia set is empty. 
For example, one can show that $\mathcal{J}(z^d)$ is empty where $d$ is a natural number.

In the theory of complex dynamical systems, Montel's theorem is well known as one of the most useful theorem to determine equicontinuity of a given family of rational maps over $\mathbb{C}$.
See \cite[Theorem $3.7$]{Miln06}.
By applying Montel's theorem, one can obtain some properties of the Julia sets of rational maps such as the facts that Julia set is uncountable and has no isolated points.

In the theory of $p$-adic dynamical systems, L-C. Hsia proved an analogue of Montel's theorem for families of rational maps over $\mathbb{C}_p$.
See \cite[Theorem $5.27$]{Silv07}.
We do not directly use Montel's theorem to prove our main theorem but we will use the following properties of the Julia sets that can be shown by Montel's theorem.

\begin{theorem}\label{montel}
Let $f : \hat{\mathbb{C}}_p \rightarrow \hat{\mathbb{C}}_p$ be a polynomial map of degree $d \geq 2$ over $\mathbb{C}_p$.
Then for any $P$ in the Julia set $\mathcal{J}(f)$ of $f$, we have
$$
\mathcal{J}(f) = \overline{\bigcup_{k \geq 0} f^{-k}(P) }.
$$
Moreover, if a non-empty and closed subset $B$ of $\mathbb{C}_p$ satisfies $f^{-1}(B) \subset B$, then $\mathcal{J}(f)$ is contained in $B$.
\end{theorem}

Note that if $f$ has a point $Q$ in $\mathbb{C}_p$ satisfying $f^{-1}(\{ Q \} ) = \{ Q \}$, then $f$ has no Julia set.
We refer \cite{Hs00} and \cite[Proposition $5.30$ and Corollary $5.32$]{Silv07} to readers.

\subsection{$p$-adic analysis}

In the field of $p$-adic complex numbers, it is easier to find zeros of polynomial maps than in the field of complex number.
See \cite{Rob00} or \cite[Theorem $5.11$]{Silv07} for proofs of the following theorem.

\begin{theorem}\label{weierstrass}
Let $f : \hat{\mathbb{C}}_p \rightarrow \hat{\mathbb{C}}_p$ be a polynomial map of degree $d \geq 1$ over $\mathbb{C}_p$ with the expansion
$$
f(z) = c_0 + c_1 (z - z_0) + \cdots + c_d (z - z_0)^d
$$ 
where $\{ c_k \}_{k = 0}^{d}$ in $\mathbb{C}_p$ with $c_d \neq 0$ and $z_0$ in $\mathbb{C}_p$.
Let $r$ be a positive real number and set 
$$
l = \max\{ l' \in \{ 0, 1, \cdots, d \} \mid | c_{l'} | \cdot r^{l'} \geq | c_{k} | \cdot r^{k} \text{ for all } k \in \{ 0, 1, \cdots, d \}  \}.
$$
Then there exist $l$ elements $\alpha_1, \alpha_2, \cdots, \alpha_l$ in $\overline{D}_r(z_0)$ and a polynomial map $g : \hat{\mathbb{C}}_p \rightarrow \hat{\mathbb{C}}_p$ over $\mathbb{C}_p$ such that $g$ has no zeros in $\overline{D}_r(z_0)$ and 
$$
f(z) = g(z) \cdot \prod_{k = 1}^{l}(z - \alpha_k).
$$
In particular, if $l = 0$, $f$ has no zeros in $\overline{D}_r(z_0)$.
\end{theorem}

\begin{proposition}\label{open}
Let $f : \hat{\mathbb{C}}_p \rightarrow \hat{\mathbb{C}}_p$ be a polynomial map of degree $d \geq 1$ over $\mathbb{C}_p$.
\begin{enumerate}

\item
The polynomial map $f$ is continuous and open on $\hat{\mathbb{C}}_p$.

\item
If $f$ has no zeros in $\overline{D}_r(z_0)$, then
$$
|f(z)| = |f(z_0)| > 0
$$
for all $z$ in $\overline{D}_r(z_0)$.
\end{enumerate}
\end{proposition}

See \cite[Corollary $5.17$ and Theorem $5.13$]{Silv07} for the proof. 
The following corollary was proven by R. Benedetto in \cite{Ben02} and it will be helpful to determine whether a given polynomial map is bijective or not.

\begin{corollary}\label{bene}
Let $f : \hat{\mathbb{C}}_p \rightarrow \hat{\mathbb{C}}_p$ be a polynomial map of degree $d \geq 1$ over $\mathbb{C}_p$, $\alpha$ be an element in $\mathbb{C}_p$, and $r$ be a positive real number.
The polynomial map $f$ is bijective from $\overline{D}_r(\alpha)$ to $\overline{D}_s(f(\alpha))$ if and only if 
$$
|f(z) - f(w)| = {s \over r} \cdot |z - w|
$$ 
for any $z$ and $w$ in $\overline{D}_r(\alpha)$.
\end{corollary}


\section{A key lemma and proofs of main theorems}

In this section, we will introduce a key lemma and prove it. 
We also prove Theorem \ref{2.1}, Theorem \ref{2.2}, and Theorem \ref{2.3}.

\subsection{A key lemma}
\begin{lemma}\label{3.2}
Let $f : \hat{\mathbb{C}}_p \rightarrow \hat{\mathbb{C}}_p$ be a polynomial map of degree $d \geq 2$ over $\mathbb{C}_p$.
Suppose that there exists a non-empty set $B$ in $\mathbb{C}_p$ satisfying the following conditions.
\begin{enumerate}
\item The polynomial map $f$ is immediately expanding on $B$.
\item The set $B$ has no critical values of $f$, that is, 
$$
B \cap f(\{ z \in \mathbb{C}_p \mid f'(z) = 0 \}) = \emptyset.
$$
\item The set $B$ is closed with respect to the choral metric.
\end{enumerate}

Then for any $w \in B$, there exist $d$ distinct elements $z_1, z_2, \cdots, z_d $ in $\mathbb{C}_p$ such that for each $k$ in $\{1, 2, \cdots, d \}$, $f(z_k) = w$.
Moreover, there exists a positive real number $\mu$ such that for any $w \in B$ and $r \in [0, \mu] \cap |\mathbb{C}_p^{\times}|$, 
$$
f | \overline{D}_{R_k}(z_k) \rightarrow \overline{D}_{r}(w)
$$
is a homeomorphism for each $k$ in $\{1, 2, \cdots, d \}$ and
$$
f^{-1}(\overline{D}_r(w)) = \bigsqcup_{k=1}^{d}\overline{D}_{R_k}(z_k)
$$
where $R_k : = r / |f'(z_k)| < r$.

\end{lemma}

\begin{proof}[Proof of Lemma \ref{3.2}]
Let us fix an arbitrary element $w$ in $B$.
Since the degree of $f$ is $d$ and $\mathbb{C}_p$ is algebraically closed, there exist $d$ elements $z_1, z_2, \cdots, z_d$ in $\mathbb{C}_p$, counted with multiplicity, such that $f(z_k) = w$ for each $k$ in $\{ 1, 2, \cdots, d \}$.
Since the set $B$ has no critical values of $f$, it is clear that the elements $z_1, z_2, \cdots, z_d$ must be distinct.

Let us construct a positive real number $\mu$, which is independent of $w$, as follows.
First, we can find a positive real number $\delta$ satisfying 
$$
\bigcup_{z: \text{ critical point} }\overline{D}_{\delta}(z) \cap f^{-1}(B) = \emptyset,
$$
since the number of critical points of $f$ in $\mathbb{C}_p$ is finite and the critical points of $f$ are in $\mathbb{C}_p \backslash f^{-1}(B)$ which is an open set with respect to $|\cdot|$.
Next we set another positive real number
$$
\mu := \min\{|l|^{1 / (l-1)} \mid l \in \{ 2, 3, \cdots, d \} \} \cdot \delta > 0.
$$
Note that for all integers $m$, we have $|m| \leq 1 $. 
Now let us show the following proposition.

\begin{proposition}\label{casethan2}
For any $z_0$ in $f^{-1}(B)$ and $l$ in $\{2, 3, \cdots, d \}$, we have
$$
|l| \cdot \left| {f^{(l)}(z_0) \over l !} \right| \cdot \delta^{l - 1}  < |f'(z_0) |
$$ 
where $f^{(l)}$ is the $l$ th derivative of $f$.
\end{proposition}
\begin{proof}[Proof of Proposition \ref{casethan2}.]
Considering the expansions of $f$ and $f'$ around $z_0$, we have
\begin{align*}
f(z) &= f(z_0) + f'(z_0) (z - z_0) + \cdots + {f^{(d)}(z_0) \over d!} (z - z_0)^d, \\
f'(z) &= f'(z_0) + 2 \cdot { f^{(2)}(z_0) \over 2 ! } (z - z_0) + \cdots + d \cdot {f^{(d)}(z_0) \over d !} (z - z_0)^{d- 1}.
\end{align*}
Moreover, it follows immediately from the construction of $\delta$ that $f$ has no critical points in $\overline{D}_{\delta}(z_0)$.
Hence, by Theorem \ref{weierstrass}, for any $l$ in $\{ 2, 3, \cdots, d \}$, we have
$$
|l| \cdot \left| {f^{(l)}(z_0) \over l! } \right| \cdot \delta^{l -1} < |f'(z_0)|.
$$
\end{proof}

In particular, this implies that 
$$
r^{l} \cdot \left| {f^{(l)}  (z_0) \over l ! } \right|  
\leq r \cdot \left| {f^{(l)}  (z_0) \over l ! } \right| \cdot \mu^{l-1} 
\leq r \cdot \left| {f^{(l)}  (z_0) \over l ! } \right| \cdot | l | \cdot \delta^{l-1} 
< r \cdot |f'(z_0)|
$$
for any $l \in \{2, 3, \cdots, d \}$ and $r \in [0, \mu] \cap |\mathbb{C}_p^{\times}|$. 

As a result, we have the following proposition.

\begin{proposition}\label{casethan1}
Suppose that $k \in \{ 1, 2, \cdots, d \}$ and $r \in [0, \mu] \cap |\mathbb{C}_p^{\times}|$. Then for any $x$ and $y$ in $\overline{D}_r(z_k)$, we have 
$$
|f(x) - f(y)| = |f'(z_k)| \cdot |x - y|.
$$
\end{proposition}
\begin{proof}[Proof of Proposition \ref{casethan1}.]
For any $k$ in $\{ 1, 2, \cdots, d\}$ and any $r$ in $[0, \mu] \cap |\mathbb{C}_p^{\times}|$, it follows from Proposition \ref{nonarchi} and Proposition \ref{casethan2} that
\begin{align*}
|f(x) - f(y)| =& |x - y| \cdot \max \{|f'(z_k)|, \left| { f^{(2)}(z_k) \over 2! } \right| \cdot |x - z_k + y - z_k|, \cdots, \\
&\qquad \qquad \qquad \left| { f^{(d)}(z_k) \over d! }  \right| \cdot |(x - z_k)^{d- 1} + (x - z_k)^{d - 2}(y - z_k) \\ 
&\qquad \qquad \qquad + \cdots + (y - z_k)^{d-1} | \}
= |f'(z_k)| \cdot |x - y|.
\end{align*}
\end{proof}

In particular, setting 
$$
R_k := r / |f'(z_k)|,
$$
we have 
$$
R_k \leq r \quad \text{and} \quad |f(x) - f(y)| = {r \over R_k} \cdot |x - y|
$$
for any $k$ in $\{1, 2, \cdots, d \}$, any $r$ in $[0, \mu] \cap |\mathbb{C}_p^{\times}|$, and any $x $ and $y$ in $\overline{D}_{R_k}(z_k)$. 
By Corollary \ref{bene}, we have that 
$$
f | \overline{D}_{R_k}(z_k) \rightarrow \overline{D}_r(w)
$$
is bijective for any $k$ in $\{ 1, 2, \cdots, d \}$. 
Moreover, by Proposition \ref{open},
the restriction map is homeomorphic for each $k$ in $\{1, 2, \cdots, d \}$.
\end{proof}


\subsection{Proof of Theorem \ref{2.1}}

In this subsection, we prove Theorem \ref{2.1} assuming Theorem \ref{2.2}, which will be proved in the next subsection.

\begin{proof}[Proof of Theorem \ref{2.1}]
Let $f : \hat{\mathbb{C}}_p \rightarrow \hat{\mathbb{C}}_p$ be a polynomial map of degree $d$ over $\mathbb{C}_p$ with $\mathcal{J}(f) \neq \emptyset$.
Suppose that there exists a $\lambda > 1$ such that $|f'(z)| > \lambda$ for any $z$ in $\mathcal{J}(f)$.

We will construct a set $B$ in $\mathbb{C}_p$ satisfying all conditions of Theorem\ref{2.2} as follows. 
Since $\mathcal{J}(f)$ has no critical points of $f$, all critical points of $f$ must be in $\mathcal{F}(f)$. 
In particular, this implies that the critical values of $f$ are also in the Fatou set $\mathcal{F}(f)$ because $\mathcal{F}(f)$ is forward invariant under $f$.
Moreover, since $\mathcal{F}(f)$ is open, we have a positive real number $\delta$ such that 
$$
\bigcup_{c \in C(f)} \overline{D}_{\delta}(c) \cap \mathcal{J}(f) = \emptyset
$$
where $C(f)$ is the set of critical points and critical values of $f$.
Without loss of generality, we may assume that the positive real number $\delta$ is in $|\mathbb{C}_p^{\times}|$. 
Let us define $B$ by 
$$
B = \bigcup_{z \in \mathcal{J}(f)}\overline{D}_{\mu}(z)
$$
where $\mu := \min\{|k|^{1 / (k - 1)} \mid k \in \{ 2, 3, \cdots, d \} \} \cdot \delta > 0$.
Note that $f$ has no critical points and critical values in $B$ because $\mu$ is less than or equal to $\delta$.
It is clear that $B$ is bounded and contains $\mathcal{J}(f)$ since $\mathcal{J}(f)$ is bounded.
Moreover, it is not hard to check that $\overline{D}_{\mu}(z)$ is contained in $B$ for any $z$ in $B$.

To complete this proof, let us check the other conditions stated in Theorem \ref{2.2}.

\begin{proposition}\label{expandingthm1}
For any $z$ in $B$, $|f'(z)| \geq \lambda.$
\end{proposition}

\begin{proof}[Proof of Proposition \ref{expandingthm1}.]
For any $z$ in $B$, there exists a $z_0$ in $\mathcal{J}(f)$ such that $z \in \overline{D}_\mu(z_0)$.
Since there is no zeros of $f'$ in $\overline{D}_{\mu}(z_0)$, it follows from Proposition \ref{open} that $|f'(z)| = |f'(z_0)| \geq \lambda$.
\end{proof}

To check the second condition, let us use the following proposition.

\begin{proposition}\label{invariantthm1}
The set $B$ is backward invariant under $f$, that is, $f^{-1}(B) \subset B$.
\end{proposition}

\begin{proof}[Proof of Proposition \ref{invariantthm1}.]
It follows from the construction of $B$ that for any $w \in B$, there exists a $w_0 \in \mathcal{J}(f)$ such that $w \in \overline{D}_{\mu}(w_0).$
By Lemma \ref{3.2}, there exist exactly $d$ elements $z_1, z_2, \cdots, z_d$ in $\mathcal{J}(f)$ such that 
$$
f^{-1}(\overline{D}_{\mu}(w_0)) = \bigsqcup_{k=1}^{d}\overline{D}_{R_k}(z_k) \quad \text{and} \quad f(z_1) = f(z_2) = \cdots = f(z_d) = w_0
$$
where $R_k = \mu / |f'(z_k)| < \mu$ for any $k$ in $\{1, 2, \cdots, d \}$.
This implies that 
$$
f^{-1}(\{ w \}) \subset f^{-1}(\overline{D}_{\mu}(w_0)) = \bigsqcup_{k = 1}^{d}\overline{D}_{R_k}(z_k) \subset \bigsqcup_{k = 1}^{d}\overline{D}_{\mu}(z_k) \subset B.
$$
\end{proof}

Hence, Theorem \ref{2.2} can be applied.
\end{proof}



\subsection{Proof of Theorem \ref{2.2}}

Now let us prove Theorem \ref{2.2}.

\begin{proof}[Proof of Theorem \ref{2.2}]

Let $f : \hat{\mathbb{C}}_p \rightarrow \hat{\mathbb{C}}_p$ be a polynomial map of the degree $d \geq 2$ over $\mathbb{C}_p$ and $B$ be a set in $\mathbb{C}_p$ and $\lambda > 1$ be a real number satisfying $|f'(z)| \geq \lambda$ for any $z \in B$.
Let $M$ and a $\delta$ be positive real numbers satisfying 
$$
B \subset \overline{D}_{M}(0) \quad \text{and} \quad \overline{D}_{\delta}(z) \subset B
$$
for any $z$ in $B$.
Without loss of generality, we may assume that $M \geq 1$ and $0 < \delta \leq 1$.

Let us define {\it the $i$ th perturbation $T_{i, \epsilon} : Poly_d(\mathbb{C}_p) \rightarrow Poly_d(\mathbb{C}_p)$ of $f$ for $\epsilon \in \mathbb{C}_p$} by
$$
T_{i, \epsilon}(f)(z) := f(z) + \epsilon \cdot z^i
$$
with $T_{i, \epsilon}(f)(\infty) = \infty$ where $0 \leq i \leq d$.

Moreover, we define {\it the set of polynomial maps preserving $\lambda$ and $B$} by 
$$
\mathcal{S} := \mathcal{S}(d, \lambda, B) := \{g \in Poly_d(\mathbb{C}_p) \mid |g'(z)| \geq \lambda (\forall z \in B) \quad \text{and} \quad g^{-1}(B) \subset B \}.
$$
Note that it is clear that $\mathcal{S}$ contains $f$ so it is non-empty.
We will complete this proof in three steps.

\paragraph{Step $1$.}
In this step, we will show that for any $i \in \{d, d-1, \cdots, 0 \}$, there exists a positive real number $\tau(i)$ such that if $g$ is in $\mathcal{S}$ and $\epsilon \in \mathbb{C}_p$ satisfies $|\epsilon| < \tau(i)$, then $T_{i, \epsilon}(g)$ is also in $\mathcal{S}$.

Let us fix an arbitrary $i$ in $\{d, d-1, \cdots, 0 \}$ and set a positive real number $\tau(i) := \tau := \min\{ \eta_1, \eta_2 \}$ for 
$$
\eta_{1} := \eta_{1}(i) := { \lambda / M^{i -1} } > 0 \quad \text{and} \quad \eta_{2} := \eta_{2}(i) := { \mu / M^{i} } > 0
$$
where $\mu = \min\{ |k|^{1 / (k-1)} \mid k \in \{2, 3, \cdots, d \} \} \cdot \delta > 0$.

\begin{proposition}\label{step0}
For any $g \in \mathcal{S}, z \in B$, and $\epsilon \in \mathbb{C}_p$ with $|\epsilon| < \tau$, we have 
$$
|T_{i, \epsilon}(g)'(z)| = |g'(z)|.
$$
\end{proposition}

\begin{proof}[Proof of Proposition \ref{step0}]
since $\epsilon \in \mathbb{C}_p$ satisfies $|\epsilon| < \tau \leq \eta_1$, it is straightforward that
$$
|i \cdot \epsilon \cdot z^{i - 1}| \leq |\epsilon| \cdot |z|^{i - 1} < {\lambda \over M^{i - 1}} \cdot M^{i - 1} = \lambda 
$$
for any $z$ in $B$.
Thus, it follows from Proposition \ref{nonarchi} that
$$
|T_{i, \epsilon}(g)'(z)| 
= |g'(z) + i \cdot \epsilon \cdot z^{i - 1}| 
= \max\{|g'(z)|, |i \cdot \epsilon \cdot z^{i - 1}| \} 
= |g'(z)|
$$
for any $z$ in $B$.
\end{proof}

Next, we prove the following proposition.

\begin{proposition}\label{step1}
For any $g \in \mathcal{S}$ and $\epsilon \in \mathbb{C}_p$ with $|\epsilon| < \tau$, we have 
$$
T_{i, \epsilon}(g)^{-1}(B) \subset B.
$$
\end{proposition}

\begin{proof}[Proof of Proposition \ref{step1}.]
Let $w$ be an element in $B$. 
Then there exist exactly $d$ elements $z_1, z_2, \cdots, z_d$ in $B$ such that 
$$
g(z_1) = g(z_2) = \cdots = g(z_d) = w
$$
since $B$ has no critical values of $g$. 
Indeed, if $B$ has a critical value of $g$, then $B$ has a critical point of $g$ because $g^{-1}(B) \subset B$. 
However, it is a contradiction to the assumption that $g$ is immediately expanding on $B$.

We will show that for each $k$ in $\{ 1, 2, \cdots, d \}$, there exists exactly one element $\tilde{z}_{k} := \tilde{z}_k(\epsilon)$ in $\overline{D}_{R_k}(z_k)$ such that $T_{i, \epsilon}(g)(\tilde{z}_{k}) = w$ where $R_k := \mu / |g'(z_k)|$.
To ease notation, we shall use ${g}_{\epsilon}$ to denote $T_{i, \epsilon}(g)$.

Fixing $k$ in $\{1, 2, \cdots, d \}$ and considering the expansion ${g}_{\epsilon}$ around $z_k$, we have
$$
{g}_{\epsilon}(z) = {g}_{\epsilon}(z_k) + {g}_{\epsilon}'(z_k)(z - z_k) + \cdots + {{g}_{\epsilon}^{(d)}(z_k) \over d ! } (z - z_k)^d.
$$
This implies that 
\begin{align*}
{g}_{\epsilon}(z) - w
&= {g}_{\epsilon}(g)(z_k) - w + {g}_{\epsilon}'(z_k)(z - z_k) + \cdots + {{g}_{\epsilon}^{(d)}(z_k) \over d ! } (z - z_k)^d \\
&= \epsilon \cdot z_k^i + {g}_{\epsilon}'(z_k)(z - z_k) + \cdots + {{g}_{\epsilon}^{(d)}(z_k) \over d ! } (z - z_k)^d.
\end{align*}

Let us see the following lemmas.

\begin{lemma}\label{step1l1}
For any $\epsilon$ in $\mathbb{C}_p$ with $|\epsilon| < \tau$, we have 
$$
|\epsilon| \cdot |z_k|^{i} < |{g}_{\epsilon}'(z_k)| \cdot R_k.
$$
\end{lemma}

\begin{proof}[Proof of Lemma \ref{step1l1}.]
It follows from Proposition \ref{step0} that 
$$
|\epsilon| \cdot |z_k|^{i} < { \mu \over M^{i}} \cdot M^i  = \mu = |g'(z_k)| \cdot R_k = |{g}_{\epsilon}'(z_k)| \cdot R_k.
$$
\end{proof}

\begin{lemma}\label{step1l2}
For any $\epsilon$ in $\mathbb{C}_p$ with $|\epsilon| < \tau$ and any $l$ in $\{2, 3, \cdots, d \}$, we have 
$$
\left| {{g}_{\epsilon}^{(l)}(z_k) \over l ! } \right| \cdot R_k^l  < |{g}_{\epsilon}'(z_k)| \cdot R_k.
$$
\end{lemma}

\begin{proof}[Proof of Lemma \ref{step1l2}.]
Considering the expansion of ${g}_{\epsilon}'$ around $z_k$, we obtain that
$$
{g}_{\epsilon}'(z) = {g}_{\epsilon}'(z_k) + 2 \cdot { {g}_{\epsilon}^{(2)}(z_k) \over 2! } (z - z_k) + \cdots + d \cdot {{g}_{\epsilon}^{(d)}(z_k) \over d!} (z - z_k)^{d - 1}.
$$
It follows from Proposition \ref{step0} that ${g}_{\epsilon}'$ has no zeros in $\overline{D}_{\delta}(z_k)$. 
Moreover, it follows from Theorem \ref{weierstrass} that
$$
\left| {{g}_{\epsilon}^{(l)} (z_k) \over l ! } \cdot l  \right| \cdot \delta^{l-1} < |{g}_{\epsilon}'(z_k)|
$$
for each $l$ in $\{2, 3, \cdots, d \}$.
Finally, we obtain that 
\begin{align*}
\left| {{g}_{\epsilon}^{(l)} (z_k) \over l ! } \right| \cdot R_k^{l} 
&= R_k \cdot \left| {{g}_{\epsilon}^{(l)}(z_k) \over l ! } \right| \cdot R_k^{l-1} 
< R_k \cdot \left| {{g}_{\epsilon}^{(l)} (z_k) \over l ! } \right| \cdot \mu^{l-1} \\
&\leq R_k \cdot \left| {{g}_{\epsilon}^{(l)} (z_k) \over l ! } \right| \cdot |l| \cdot \delta^{l-1} 
< |{g}_{\epsilon}'(z_k)| \cdot R_k.
\end{align*}
\end{proof}

Hence, by Theorem \ref{weierstrass}, the polynomial map ${g}_{\epsilon}(z) - w$ has a unique zero in $\overline{D}_{R_k}(z_k)$.
Moreover, it is easy to check that each disk $\overline{D}_{R_k}(z_k)$ is disjoint.
Let us denote the unique zero of ${g}_{\epsilon}(z) - w$ in $\overline{D}_{R_k}(z_k)$ by $\tilde{z}_k$ for each $k$ in $\{1, 2, \cdots, d \}$.
In particular, this implies $\tilde{z}_{k}$ in $B$ since 
$$
\tilde{z}_{k} \in \overline{D}_{R_k}(z_k) \subset \overline{D}_{\mu}(z_k) \subset \overline{D}_{\delta}(z_k) \subset B
$$
for each $k$ in $\{1, 2, \cdots, d \}$.
\end{proof}

By Proposition \ref{step0} and \ref{step1}, $T_{i, \epsilon}(\mathcal{S}) \subset \mathcal{S}$ for any $\epsilon \in \mathbb{C}_p$ satisfying $|\epsilon| < \tau(i)$.
In particular, this implies that 
$$
T_{0, \epsilon_0} \circ T_{1, \epsilon_1} \circ \cdots \circ T_{d, \epsilon_d}(\mathcal{S}) \subset \mathcal{S}
$$ 
if a sequence $\{ \epsilon_k \}_{k = 0}^{d} \subset \mathbb{C}_p$ satisfies $\max\{ |\epsilon_0|, |\epsilon_1|, \cdots, |\epsilon_d| \} \leq \min\{\tau(0), \tau(1), \cdots, \tau(d) \}$.

Let us close this step with the following proposition.

\begin{proposition}\label{step3}
Let $g$ be a element in $\mathcal{S}$ and $i$ be a natural number in $\{d, d-1, \cdots, 0 \}$.
If $\epsilon \in \mathbb{C}_p$ satisfies $|\epsilon| < \tau(i)$, then we have
$$
|T_{i, \epsilon}(g)(z) - g(z) | \leq \mu
$$
for any $z \in B$.
In particular, if $\{ \epsilon_k \}_{k = 0}^{d} \subset \mathbb{C}_p$ satisfies $\max\{ |\epsilon_0|, |\epsilon_1|, \cdots, |\epsilon_d| \} \leq \min\{\tau(0), \tau(1), \cdots, \tau(d) \}$, then 
$$
|T_{d, \epsilon_d} \circ T_{d-1, \epsilon_{d-1}} \circ \cdots \circ T_{0, \epsilon_{0}}(f) (z) - f(z)| \leq \mu 
$$
for any $z \in B$.
\end{proposition}

\begin{proof}[Proof of Proposition \ref{step3}]
It is straightforward that 
$$
|T_{i, \epsilon}(g)(z) - g(z) | = |g(z) + \epsilon \cdot z^{i} - g(z)| 
= |\epsilon| \cdot |z|^i \leq \tau(i) \cdot M^i \leq {\mu \over M^i} \cdot  M^i = \mu
$$
for any $z \in B$.
Moreover, for any $\{ \epsilon_k \}_{k = 0}^{d} \subset \mathbb{C}_p$ satisfying $\max\{ |\epsilon_0|, |\epsilon_1|, \cdots, |\epsilon_d| \} \leq \min\{\tau(0), \tau(1), \cdots, \tau(d) \}$, we have 
\begin{align*}
|f_{d}(z) - f(z)| 
&= |f_{d}(z) - f_{d-1}(z) + f_{d-1}(z) - f_{d-2}(z) + \cdots - f_{0}(z) + f(z) | \\
&\leq \max \{ |f_{d}(z) - f_{d-1}(z)|, | f_{d-1}(z) -  f_{d-2}(z)|, \cdots, | f_{0}(z) - f(z)|\} 
\leq \mu
\end{align*}
for any $z \in B$ where $f_k(z) := T_{k, \epsilon_k} \circ T_{k-1, \epsilon_{k-1}} \circ T_{0, \epsilon_0}(f) (z)$ for each $k \in \{d, d-1, \cdots, 0 \}$.
\end{proof}

\paragraph{Step $2$.}
In this step, we will construct a conjugacy between the dynamics $\mathcal{J}(f)$ and $\mathcal{J}(g)$ if $g \in \mathcal{S}$ satisfies $|g(z) - f(z)| \leq \mu$ for any $z \in B$ where $\mu := \min \{ |l|^{1/(l - 1)} \mid l \in \{ 2, 3, \cdots, d \} \} \cdot \delta$.
Let us begin with a construction of sets.

\paragraph{$\bullet$ The construction of a family $\{ \Omega_{k}(g) \}_{k \geq 0}$ of sets in $\mathbb{C}_p$.}

For any $g \in \mathcal{S}$, we define a family $\{ \Omega_{k}(g) \}_{k \geq 0}$ of sets in $\mathbb{C}_p$ by 
$$
\Omega_{0}(g) := B, \quad \Omega_{1}(g) := g^{-1}(\Omega_{0}(g)), \quad \cdots, \quad \Omega_{k + 1}(g) := g^{-1}(\Omega_{k}(g)), \quad \cdots
$$
for each $k \geq 0$. 

Moreover, we define the limit set of $\{ \Omega_{k}(g) \}_{k \geq 0}$ by 
$$
\Omega_{\infty}(g) := \bigcap_{k = 0}^{\infty}\Omega_{k}(g).
$$

Since $B$ is backward invariant under $g$, it is clear that 
$$
\Omega_{\infty}(g) \quad \subset \quad \cdots \quad \subset \quad \Omega_{k + 1}(g) \quad \subset \quad \Omega_{k}(g) \quad \subset \quad \cdots \quad \subset \quad \Omega_{0}(g) \quad = \quad B
$$
for each $k \geq 0$.

To ease notation, we denote the sets $\Omega_{k}(f)$ by $\Omega_{k}$ for each $k \geq 0$ and $k = \infty$.
Then it is clear that 
$$
\mathcal{J}(f) \quad \subset \quad \Omega_{\infty} \quad \subset \quad \cdots \quad \subset \quad \Omega_{k + 1} \quad \subset \quad \Omega_{k} \quad \subset \quad \cdots \quad \subset \quad \Omega_{0}.
$$
In particular, this implies that the limit set $\Omega_{\infty}$ is non-empty.

\paragraph{$\bullet$ The construction of a family $\{ h_k : \Omega_k \rightarrow \Omega_k(g) \}_{k \geq 0}$ of homeomorphisms.}

Let us fix $g \in \mathcal{S}$ satisfying $|f(z) -g(z)| \leq \mu$ for any $z \in B$ and define $h_0$ as the identity map on $B$.
It is clear that $h_0 : \Omega_{0} \rightarrow \Omega_{0}(g)$ is a homeomorphism.
We will construct a family of homeomorphism $\{ h_k : \Omega_k \rightarrow \Omega_k(g) \}_{k \geq 1}$ of homeomorphisms, inductively.

To define the map $h_1: \Omega_1 \rightarrow \Omega_{1}(g)$, let us prove the following proposition.

\begin{proposition}\label{step2}
Let $g \in \mathcal{S}$ be an element satisfying $|g(z) - f(z)| \leq \mu$ for any $z \in B$.
Then for any $z$ in $\Omega_1$, there exists a unique $w := w(z)$ in $g^{-1}( \{ h_0 \circ f(z) \} )$ such that
$$
|w - z| \leq { \mu / \lambda }.
$$
\end{proposition}

\begin{proof}[Proof of Proposition \ref{step2}.]

Let $z$ be an element in $\Omega_1$.
To ease notation, we shall denote 
$$
g^{-1}(\{ h_0 \circ g(z) \}) : = \{ z_1, z_2, \cdots, z_d \} \quad \text{and} \quad g^{-1}(\{ h_0 \circ f(z) \}) := \{w_1, w_2, \cdots, w_d \}.
$$
It follows from $|f(z) - g(z)| \leq \mu$ that $h_0 \circ f(z) = f(z) \in \overline{D}_{\mu}(g(z)) \subset B$.

Let us set 
$$
R_1 = \mu / |g'(z_1)|, \quad R_2 = \mu / |g'(z_2)|, \quad \cdots, \quad R_d = \mu / |g'(z_d)|.
$$
Then it follows from Lemma \ref{3.2} that 
$$
g^{-1}(\overline{D}_{\mu}(g(z))) = \bigsqcup_{k = 1}^{d} \overline{D}_{R_k}(z_k)
$$
and $g \mid \overline{D}_{R_k}(z_k) \rightarrow \overline{D}_{\mu}(g(z))$
is a homeomorphism for each $k$ in $\{1, 2, \cdots, d \}$.

Therefore, we obtain that
$$
\{w_1, w_2, \cdots, w_d \} = g^{-1}(\{ h_0 \circ f(z) \}) \subset g^{-1}(\overline{D}_{\mu}(g(z))) = \bigsqcup_{k = 1}^{d} \overline{D}_{R_k}(z_k).
$$
This implies that there exists a unique $w_k \in \overline{D}_{R_k}(z_k) \subset \overline{D}_{\mu / \lambda}(z_k)$ for each $k$ in $\{1, 2, \cdots, d \}$.
Indeed, if there exist two distinct elements $w_i$ and $w_j \in \overline{D}_{R_k}(z_k)$ for some $k$ in $\{1, 2, \cdots, d \}$, we have that $g(w_i) = g(w_j) = f(z) \in \overline{D}_{\mu}(g(z))$. However, it is a contradiction to the fact that 
$g \mid \overline{D}_{R_k}(z_k) \rightarrow \overline{D}_{\mu}(g(z))$ is injective. See Lemma \ref{3.2}.
\end{proof}

Let us define a map $h_1 : \Omega_1 \rightarrow \Omega_{1}(g)$ by $w := w(z)$ determined in Proposition \ref{step2}.
It follows from the construction of $h_1 : \Omega_1 \rightarrow \Omega_{1}(g)$ that $h_1 : \Omega_1 \rightarrow \Omega_{1}(g)$ is continuous and open on $\Omega_1$. 
Moreover, it satisfies $h_0 \circ f = g \circ h_{1}$ on $\Omega_1$.
Furthermore, one can construct the inverse map of $h_1 : \Omega_1 \rightarrow \Omega_{1}(g)$, similarly.
This implies that $h_1 : \Omega_1 \rightarrow \Omega_{1}(g)$ is a homeomorphism.

Let us assume that we have constructed a family $\{h_m : \Omega_{m} \rightarrow \Omega_{m}(g) \}_{m = 0}^{n}$ of homeomorphisms such that 
$$
| h_{m + 1}(z) - h_{m}(z)| < \mu / \lambda^{m + 1} 
$$
and $h_{m} \circ f = g \circ h_{m + 1}$ on $\Omega_{m + 1}$ for each $m$ in $\{ 0, 1, \cdots, n - 1 \}$ where $n \geq 1$.
Now we have the following claim.
The proof is similar to the proof of Proposition \ref{step2} so we omit it.

\begin{proposition}\label{step2p2}
For each $z \in \Omega_{n + 1}$, there exists a unique $w := w(z) \in g^{-1}( \{ h_{n} \circ f (z) \} )$ such that 
$$
|w - h_{n}(z)| \leq \mu / \lambda^{n + 1}.
$$
\end{proposition}

Let us define $h_{l + 1} : \Omega_{l + 1} \rightarrow \Omega_{l + 1}(g)$ by $w := w(z)$ determined in Proposition \ref{step2p2}.
By the construction of $h_{l + 1} : \Omega_{l + 1} \rightarrow \Omega_{l + 1}(g)$, it is clear that $h_{l + 1} : \Omega_{l + 1} \rightarrow \Omega_{l + 1}(g)$ is continuous and open on $\Omega_{l + 1}$. 
Moreover, it satisfies $h_{l} \circ f = g \circ h_{l + 1}$ on $\Omega_{l + 1}$.
Furthermore, one can construct the inverse map of $h_{l+1} : \Omega_{l + 1} \rightarrow \Omega_{l + 1}(g)$ so $h_{l + 1} : \Omega_{l + 1} \rightarrow \Omega_{l + 1}(g)$ is a homeomorphism.

Therefore, we can get a family $\{ h_{k} : \Omega_{k} \rightarrow \Omega_{k}(g) \}_{k \geq 0}$ of homeomorphisms satisfying
$$
|h_{k + 1}(z) - h_{k}(z)| < \mu / \lambda^{k + 1} \quad \text{and} \quad h_k \circ f = g \circ h_{k + 1}
$$
on each $\Omega_{k + 1}$.

\paragraph{$\bullet$ The construction of the limit maps $ h_\infty : \Omega_\infty \rightarrow \Omega_{\infty}(g)$.}

We define the limit map $h_{\infty}$ of a family $\{ h_{k} : \Omega_k \rightarrow \Omega_{k}(g) \}_{k \geq 0}$ of the homomorphisms by
\begin{align*}
h_\infty : \Omega_\infty &\rightarrow \Omega_{\infty}(g) \\
z &\mapsto \lim_{k \rightarrow \infty} h_k(z).
\end{align*}

Then we have the following proposition.

\begin{proposition}\label{step2p3}
The limit map $h_\infty : \Omega_{\infty} \rightarrow \Omega_{\infty}(g)$ is well-defined.
Moreover, it is a homeomorphism.
Furthermore, it satisfies $h_{\infty} \circ f = g \circ h_{\infty}$ 
on $\Omega_{\infty}$ and $h_\infty(\mathcal{J}(f)) = \mathcal{J}(g).$
\end{proposition}

\begin{proof}[Proof of Proposition \ref{step2p3}]

The sequence $\{ h_k : \Omega_{k} \rightarrow \Omega_{k}(g) \}_{k \geq 0}$ of homeomorphisms is convergent since it is Cauchy.
Moreover, the limit $h_{\infty} : \Omega_{\infty} \rightarrow \Omega_{\infty}(g)$ must be a homeomorphism since $\Omega_{\infty}$ and $\Omega_{\infty}(g)$ are closed subsets in $\mathbb{C}_p$.

Since every polynomial maps over $\mathbb{C}_p$ is continuous, we obtain that 
\begin{align*}
h_{\infty} \circ f(z) 
&= \lim_{k \rightarrow \infty}h_{k} \circ f(z)
= \lim_{k \rightarrow \infty}g \circ h_{k}(z) \\
&= g \circ \lim_{k \rightarrow \infty} h_{k}(z)
= g \circ h_{\infty} (z).
\end{align*}

Finally, let us check that $h_\infty(\mathcal{J}(f)) = \mathcal{J}(g)$.
To show that we use the notion of {\it repelling periodic point}.
We call a point $P$ in $\mathbb{C}_p$ {\it a repelling periodic point of $f$} if there exists a $m \in \mathbb{N}$ such that $f^{m}(P) = P$ and $|(f^{m})'(P)| > 1$.
It is well known that every repelling periodic point is in the Julia set.
See \cite[Proposition $5.20$ $(2)$]{Silv07}.

Let us begin with the existence of a repelling periodic point of $f$.
\begin{lemma}\label{step2l1}
There exists a repelling periodic point $P$ of $f$ in $B$.
\end{lemma}
\begin{proof}[Proof of Lemma \ref{step2l1}]
It is well-known that the Julia set $\mathcal{J}(f)$ of $f$ is contained in the closure of the set of periodic points of $f$ with respect to the chordal metric. 
See \cite[Theorem 3.1]{Hs00} or \cite[Theorem 5.37]{Silv07}.
Since $B$ is bounded, we have that $\mathcal{J}(f)$ is contained in the closure of the set of periodic points of $f$ with respect to $|\cdot|$.
In particular, for any $z \in \mathcal{J}(f)$, there exists a periodic point $P$ of $f$ in $\overline{D}_{\delta}(z) \subset B$.
Let us denote the minimal period of $P$ by $m$ and show that the set $\{P, f(P), \cdots, f^{m-1}(P) \}$ is invariant under $f$.
Indeed, since $B$ is backward invariant under $f$ and $f^{k}(f^{m-k}(P)) = f^{m}(P) = P$, we have $f^{m - k}(P) \in f^{-k}(\{ P \}) \subset B$ for each $k$ in $\{1, 2, \cdots, m-1 \}$. 
Since $f$ is immediately expanding on $B$, we have 
$|(f^{m})'(P)| \geq \lambda^m > 1$.
\end{proof}

Let $P$ be a repelling periodic point of $f$ in $B$ so $P$ is contained in the Julia set $\mathcal{J}(f)$ of $f$.
Thus, we obtain that $P$ is in $\Omega_{\infty}$.
Moreover, we have that $g^{m}(h_{\infty}(P)) = h_{\infty}(f^{m}(P)) = h_{\infty}(P)$ where $m$ is the minimal period of $P$.
In particular, this implies that $h_{\infty}(P)$ is a periodic point of $g$.
Since $h_{\infty}(P)$ is an element of $B$ and $g$ is immediately expanding on $B$, we have $h_{\infty}(P)$ is a repelling periodic point of $g$ so $\mathcal{J}(g)$ is non-empty.

Moreover, by Proposition \ref{montel}, we see that
\begin{align*}
h_\infty(\mathcal{J}(f)) 
&= h_\infty(\mathcal{J}(\tilde{f})) 
= h_\infty(\overline{\bigcup_{k \geq 0} {\tilde{f}}^{-k}(\{ P \})}) \\
&= \overline{\bigcup_{k \geq 0} h_{\infty} \circ {\tilde{f}}^{-k}(\{ P \})}
= \overline{\bigcup_{k \geq 0} \tilde{g}^{-k} (\{ h_{\infty} (P) \})} \\
&= \mathcal{J}(\tilde{g})
= \mathcal{J}(g)
\end{align*}
where $\tilde{f} := f^{m}$ and $\tilde{g} := g^{m}$.
\end{proof}

\paragraph{Step $3$.}
In this step, we will summarise our discussion and conclude that $f$ is $J$-stable.
Let $U$ be the set of polynomial maps $b_d \cdot z^d + b_{d-1} \cdot z^{d-1} + \cdots + b_0$ in $Poly_d(\mathbb{C}_p)$ satisfying 
$$
\max_{k \in \{0, 1, \cdots, d \}} \{ |a_k - b_k| \} < \min_{k' \in \{d, d-1, \cdots, 0 \}}\{\tau(k') \}
$$ 
where $f(z) = a_d \cdot z^d + a_{d-1} \cdot z^{d-1} + \cdots + a_0$ and we consider $U$ as a neighborhood of $f$ in $Poly_d(\mathbb{C}_p)$.
Note that if $g$ be an element in $U$ then there exists a $\{ \epsilon_0, \epsilon_1, \cdots, \epsilon_d \} \subset \mathbb{C}_p$ such that 
$$
g(z) = T_{d, \epsilon_d} \circ T_{d-1, \epsilon_{d-1}} \circ \cdots \circ T_{0, \epsilon_{0}}(f) (z).
$$
By Proposition \ref{step3}, we have 
$$
|g(z) - f(z)| = |T_{d, \epsilon_d} \circ T_{d-1, \epsilon_{d-1}} \circ \cdots \circ T_{0, \epsilon_{0}}(f)(z) - f(z)| \leq \mu
$$
for any $z \in B$.
Therefore, by Step $2$, there exists a homeomorphism 
$$
h_\infty : \mathcal{J}(f) \rightarrow \mathcal{J}(g)
$$
such that $h_\infty \circ f = g \circ h_{\infty}$ on $\mathcal{J}(f)$.


\end{proof}

\subsection{Proof of Theorem \ref{2.3}}

Finally, we prove Theorem \ref{2.3}.
\begin{proof}[Proof of Theorem \ref{2.3}]

Let us fix $c$ in $\mathbb{C}_p$ with $|c| > 1$. 
The following lemma is essential in this proof.

\begin{lemma}\label{lemma1.3}
The Julia set $\mathcal{J}(f_c)$ of $f_c$ is non-empty. Moreover, it is contained in $\{z \in \mathbb{C}_p \mid |z| = |c|^{1/d} \}$. In particular, this implies that $f_c$ is immediately expanding.
\end{lemma}

\begin{proof}[Proof of Lemma \ref{lemma1.3}]
To ease notation, we shall use $U$ to denote the set $\{z \in \mathbb{C}_p \mid |z| = |c|^{1/d } \}$.
It is trivial that $U$ is non-empty, bounded, and closed.
Moreover, one can easily check that $f_{c}^{-1}(U) \subset U$.
As an application of Montel's theorem, see \cite[Theorem $5.27$]{Silv07}, one can check that $U$ contains $\mathcal{J}(f_c)$.
Now let us find a fixed point of $f$ in $U$. 
Considering the polynomial map $g(z) : = f_c(z) - z = z^d - z + c$, we obtain that there exists $d$ zeros in $\mathbb{C}_p$, counted with multiplicity.
Let us choose an arbitrary zero of $g$ and we denote it by $w$.
Let us show that $|w|$ must be equal to $|c|^{1/d}$.
Indeed, if we assume that $|w| < |c|^{1 / d }$,  it follows from Proposition \ref{nonarchi} that 
$$
0 = |g(w)| = |w^d - w + c| = |c| \neq 0. 
$$
On the other hand, if we assume that $|w| > |c|^{1/ d }$, it also follows from Proposition \ref{nonarchi} that
$$
0 = |g(w)| = |w^d - w + c| = |w|^d > |c| \neq 0.
$$
Both of them implies  a contradiction so $|w|$ must be equal to $|c|^{1/d}$. 
Moreover, it is clear that for any $z$ in $U$, we have 
$|f_c'(z)| = |d \cdot z^{d-1}| = |c|^{(d - 1) / d} > 1$.
This implies $w$ is a repelling fixed point so $\mathcal{J}(f_c)$ is non-empty. 
See \cite[Proposition $5.20$ (2)]{Silv07}.
\end{proof}

Lemma \ref{lemma1.3} implies that it satisfies the condition of Theorem \ref{2.2}.
\end{proof}



\section{An application of the main theorems}

In this section, we will see an application of Theorem \ref{2.3}. 
We first recall the symbolic dynamical system. 
After then, we will see the Julia sets of some polynomial maps corresponding to the symbolic dynamical system.

\subsection{The symbolic dynamical system}

We define {\it the sequence space} $\Sigma$ by 
$$
\{ (s_0s_1\cdots) \mid \forall i \geq 0, s_i = 0 \text{ or } 1 \},
$$
and {\it the shift map $\sigma$ on $\Sigma$} by 
$$
\sigma (s_0s_1\cdots) = (s_1s_2\cdots).
$$
We consider 
$$
d(s, t) = \sum_{k=0}^{\infty} {|s_k - t_k | \over 2^{k} }
$$
for any elements $s = (s_0s_1\cdots)$ and $t =(t_0t_1\cdots)$ as a metric on $\Sigma$.
One can easily check the that the symbolic space $\Sigma$ is disconnected and compact with respect to the metric $d$. Moreover, the shift map $\sigma$ is continuous with respect to $d$.

\subsection{Julia sets with symbolic dynamics}
\begin{proposition}\label{4.1}
Let $p$ be an odd prime number and $F$ be a polynomial map over $\mathbb{C}_p$ defined by
$$
F(z) := {z (z - 1) \over p}.
$$
Then, the Julia set $\mathcal{J}(F)$ of $F$ is non-empty. 
Moreover, there exists a homeomorphism $h : \mathcal{J}(F) \rightarrow \Sigma$ such that 
$$
h \circ F= \sigma \circ h
$$
on $\mathcal{J}(F)$.
In particular, $\mathcal{J}(F)$ is compact.
\end{proposition}

We refer \cite[Corollary $5.25$]{Silv07} to readers.

\begin{corollary}\label{4.2}
Let $p$ be a prime number and consider a family of polynomial maps 
$$
f_{c}(z) = z^2 + c
$$
where $c \in \mathbb{C}_p$.
If a parameter $c \in \mathbb{C}_p$ satisfies that
$$
|c - \gamma | < p
$$
where $\gamma := - 1 / (2p) - 1 / (4p^2) $,
then there exists a homeomorphism $h : \mathcal{J}(f_c) \rightarrow \Sigma$ such that 
$$
h \circ f_c = \sigma \circ h
$$
on the Julia set $\mathcal{J}(f_c)$ of $f_c$.
In particular, $\mathcal{J}(f_c)$ is compact.
\end{corollary}
\begin{proof}[Proof of Corollary \ref{4.2}]
By Theorem \ref{2.3}, if $c$ satisfies $|\gamma - c| < |\gamma|^{1/2} = p$, there exists a homeomorphism $h_1 : \mathcal{J}(f_c) \rightarrow \mathcal{J}(f_{\gamma})$ such that $h_1 \circ f_c = f_{\gamma} \circ h_1$ on $\mathcal{J}(f_c)$.

Moreover, considering a map
$$
h_2(z) : =  p \cdot z + 1 / 2 ,
$$
we have that $h_2 : \mathcal{J}(f_\gamma) \rightarrow \mathcal{J}(F)$ is a homeomorphism and satisfies $h_2 \circ f_\gamma = F \circ h_2 $ on $\mathcal{J}(f_\gamma)$ where $F$ is the polynomial map defined in Proposition $4.1$.

On the other hand, by Proposition $4.1$, there exists a homeomorphism $h_3 : \mathcal{J}(F) \rightarrow \Sigma$ such that $h_3 \circ F = \sigma \circ h_3$ on $\mathcal{J}(F)$. 
That is, the following commutative diagrams are obtained.
\begin{align*}
\begin{CD}
\mathcal{J}(f_c) @> h_1 >> \mathcal{J}(f_\gamma) @> h_2 >> \mathcal{J}(F) @> h_3 >> \Sigma \\
@V{f_c}VV @VV{f_\gamma}V @VV{F}V @VV{\sigma}V \\
\mathcal{J}(f_c) @> h_1 >> \mathcal{J}(f_\gamma) @> h_2 >> \mathcal{J}(F) @> h_3 >> \Sigma.
\end{CD}
\end{align*}
Setting $h := h_3 \circ h_2 \circ h_1 : \mathcal{J}(f_c) \rightarrow \Sigma$, we obtain a homeomorphism satisfying $h \circ f_c = \sigma \circ h$ on $\mathcal{J}(f_c)$.
\end{proof}
%



\section*{Acknowledgement}

The author would like to thank Tomoki Kawahira for his advice.


\small{

}

\end{document}